\documentclass[a4paper,10pt]{article}
\usepackage[utf8x]{inputenc}
\usepackage{graphicx, color,url}
\usepackage{amsmath,amssymb,amsthm,cleveref,comment,cite}
\usepackage[ruled,vlined,linesnumbered,noend,lined]{algorithm2e}

\newcommand{\Cay}{\textrm{Cay}}
\newcommand{\dCay}{\overrightarrow{\textrm{Cay}}}

\newtheorem{thm}{Theorem}[section]

\newtheorem{obs}[thm]{Observation}

\newtheorem{lem}[thm]{Lemma}

\newtheorem{conj}[thm]{Conjecture}

\theoremstyle{definition}
\newtheorem{defi}{Definition}[section]

\title{What is and is not inside a Cayley graph?}
\author{Kolja Knauer\thanks{Departament de Matem\'{a}tiques i Inform\'{a}tica, Universitat de Barcelona, Centre de Recerca Matemàtica, Barcelona, Spain.} \and Alvaro Soto Gomez\thanks{Departament de Matem\'{a}tiques i Inform\'{a}tica, Universitat de Barcelona, Barcelona, Spain.}}
\date{}

\begin{document}

\maketitle

\begin{abstract}
In this note we show that there is a cubic graph of girth $5$ that is not a subgraph of any minimal Cayley graph. On the other hand, we show that any Generalized Petersen Graph $G(n,k)$ with $\gcd(n,k)=1$ is an induced subgraph of a minimal Cayley graph. These results give insights into two comments of László Babai in [L. Babai, \emph{Automorphism groups, isomorphism, reconstruction}.
Graham, R. L. (ed.) et al., Handbook of combinatorics. Vol. 1-2, 1994].
\end{abstract}

\section{Introduction}
Given a finite group $\Gamma$ and a \emph{connection set} $ S\subseteq\Gamma$, the (undirected, right) \emph{Cayley graph} $\Cay(\Gamma,  S)$ has vertex set $\Gamma$, where two elements $a, b \in \Gamma$ form an edge if $a^{-1}b \in  S$. Note that we consider Cayley graphs as undirected even if the connection set is not inverse-closed.

A Cayley graph is \emph{minimal} (also known as irreducible) if $ S$ is an inclusion-minimal generating set of $\Gamma$. Minimal Cayley graphs arise naturally: a famous and open problem, often attributed to Lovász, is (equivalent to) the question of whether minimal Cayley graphs are Hamiltonian; see~\cite[Section 4]{PR09}.
Other contexts in which minimal Cayley graphs appear include the study of the genus of a group~\cite{Whi01} and problems related to sensitivity~\cite[Questions 7.7 and 8.2]{GMK22}. Various aspects of these graphs—sometimes for specific groups—and their distinguishing features compared to general Cayley graphs have been explored in~\cite{MS20,SS09,LZ01,HI96,God81,FM16}.

A major question, raised by Babai~\cite{B78,B95}, with only partial recent results~\cite{GMK25,DHY24} is whether minimal Cayley graphs have bounded chromatic number. One way to construct minimal Cayley graphs of large chromatic number, would be to take a known family of graphs of large chromatic number and show that its members are subgraphs of minimal Cayley graphs. However, while every graph is an induced subgraph of some Cayley graph~\cite{BS85,GI87,Bab78}, this is far from true for minimal Cayley graphs. For instance, neither the complete graph on four vertices minus an edge, $K_4 - e$, nor the complete bipartite graph $K_{3,5}$ can be subgraphs of a minimal Cayley graph~\cite{B78}. Moreover, for every $g \geq 3$, there exists a graph of girth $g$ and maximum degree $100$ that is not a subgraph of any minimal Cayley graph~\cite{Spe83}.
A natural question is how much the degree in the above result can be lowered. We quote from~\cite{B95}:
\begin{quote}
It is not known whether or not such excluded subgraphs of girth 5 and degree 3 exist [...] 
The Petersen graph is a subgraph of a minimal Cayley graph of
a group of order 20.
\end{quote}

In this note we give an answer to the first part of the above statement, i.e., we provide a graph of girth 5 and degree 3 that is not a subgraph of a minimal Cayley graph. 
Further we extend the second part of the above statement, by showing that any Generalized Petersen Graph $G(n,k)$ with $\gcd(n,k)=1$ is an induced subgraph of a minimal Cayley graph.

\section{No-lonely-color colorings of digraphs}
It was known since~\cite{B78} that subgraphs of minimal Cayley graphs admit so-called \emph{no-lonely-color colorings} of their edges. However, this necessary property was not sufficient to find a cubic girth $5$ graph, that is not a subgraph of a Cayley graph. 
Hence, we need a stronger necessary condition. This is what we provide in this section.






For a group $\Gamma$ and subset $ S$ the Cayley digraph $\dCay(\Gamma,S)$ has vertex set $\Gamma$, where two elements $a, b \in \Gamma$ form an arc $(a,b)$ if $a^{-1}b \in  S$. Note that \emph{antiparallel} arcs, i.e., of the form $(a,b)$ and $(b,a)$, may occur. We say that $D$ is a \emph{minimal Cayley digraph} if there is an inclusion-minimal generating set if sone group $\Gamma$ such that $D\cong \dCay(\Gamma,S)$. Clearly,
\begin{obs}\label{obs:diundi}
    A digraph $D$ is subdigraph of a minimal Cayley digraph if and only if the underlying graph $G$ of $D$ is subgraph of a minimal Cayley graph.
\end{obs}

Denote by $W=(a^{\sigma(1)}_1, \ldots, a^{\sigma(\ell)}_{\ell})$ a walk in $D$, where the sign $\sigma(i)\in \{+1,-1\}$ denotes wether arc $a_i$ is traversed in forward or backward direction respectively. Finally, denote by $\underline{W}$ the subdigraph of $D$ on the set of arcs $\{a_1, \ldots, a_{\ell}\}$.

\begin{defi}\label{def:dilonely}
An arc-coloring $f:A\to S$ of a digraph $D=(V,A)$ is called \emph{faithful no-lonely-color} if \begin{enumerate}
    \item for every $i\in S$ the subdigraph $f^{-1}(i)$ has maximum outdegree and maximum indegree at most $1$
    \item for every cycle $C\subseteq A$ of $D$ and $i\in S$ we have $|C\cap f^{-1}(i)|\neq 1$
    \item for walks $W=(a^{\sigma(1)}_1, \ldots, a^{\sigma(\ell)}_{\ell}), W'=(b^{\tau(1)}_1, \ldots, b^{\tau(\ell)}_{\ell})$ we have $f(a_i)=f(b_i)$ and $\sigma(i)=\tau(i)$ for every $i\in[\ell]$, then $\underline{W}\cong \underline{W'}$.
\end{enumerate}
\end{defi}

\begin{obs}\label{obs:reverse_mono}
Let $f:A\to  S$ be a faithful no-lonely-coloring of $D=(V,A)$ then for any $i\in S$
\begin{enumerate}
    \item  reversing the orientation of $f^{-1}(i)$ yields a digraph $D^i$ for which $f$ also is a faithful no-lonely-coloring.
    \item all cycles in $f^{-1}(i)$ are directed (by Definition~\ref{def:dilonely}.1.) and monochromatic (by Definition~\ref{def:dilonely}.3.).
\end{enumerate}
\end{obs}

\begin{lem}\label{lem:dinolo}
If $D=(V,A)$ is subdigraph of a minimal Cayley digraph $\dCay(\Gamma,S)$, then assigning via $f:A\to S$ to each $a\in A$ the corresponding element of $S$ is a faithful no-lonely-color.
\end{lem}
\begin{proof}
We show the properties one by one:
\noindent 1. For every $i\in S$ the subdigraph $f^{-1}(i)$ consists of the edges corresponding to one element of $c\in S$. Hence, every vertex $g$ can have at most one outneighbor $vc$ and one in-neighbor $vc^{-1}$.

\noindent 2. By Observation~\ref{obs:diundi} the underlying graph of $D$ is subgraph of a minimal Cayley graph. By Babai~\cite{B78} the coloring $f$  satisfies~2.

\noindent 3. Consider for walks $W=(a^{\sigma(1)}_1, \ldots, a^{\sigma(\ell)}_{\ell}), W'=(b^{\tau(1)}_1, \ldots, b^{\tau(\ell)}_{\ell})$ we have $f(a_i)=f(b_i)$ and $\sigma(i)=\tau(i)$ for every $i\in[\ell]$. Now let $g$ and $h$ be the starting vertices of $W$ and $W'$, respectively. It is well-known and easy to see that in $\dCay(\Gamma,S)$ the left-multiplication with $hg^{-1}$ is a digraph automorphism that preserves arc-colors and in particular maps $g$ to $h$, see e.g.~\cite[Chapter 7.3]{KK19}. Since $f(a_i)=f(b_i)$ and $\sigma(i)=\tau(i)$ for every $i\in[\ell]$, this automorphism induces an isomorphism from $\underline{W}$ to $\underline{W'}$.
\end{proof}

It follows directly that
\begin{obs}\label{obs:matching}
Let $D=(V,A)$ be a subgraph of a minimal Cayley digraph $\dCay(\Gamma, S)$ and $f:A\to S$ the faithful no-lonely-color coloring coming from Lemma~\ref{lem:dinolo}. If for some $i\in S$ there is a cycle of length $2$ in $f^{-1}(i)$, then each component of $f^{-1}(i)$ has at most two vertices and augmenting it to a cycle of length $2$, produces a $D'$, such that $D\subseteq D'\subseteq \dCay(\Gamma, S)$ with the correspondingly extended $f$.
\end{obs}

We turn our attention to undirected graphs.

\begin{defi}\label{def:lonely}
An edge-coloring $f:E\to S$ of a graph $G=(V,E)$ is called \emph{faithful no-lonely-color} if \begin{enumerate}
    \item for every $i\in S$ the subgraph $f^{-1}(i)$ has maximum degree $2$,
    \item for every cycle $C\subseteq E$ of $G$ and $i\in S$ we have $|C\cap f^{-1}(i)|\neq 1$,
    \item for every $i\in S$ all cycles in the subgraph $f^{-1}(i)$ are of the same length.
\end{enumerate}
\end{defi}

If $G$ is subgraph of a minimal Cayley graph, then by Observation~\ref{obs:diundi} $G$ is the underlying graph of $D$ a subdigraph of a minimal Cayley digraph. Now, we can apply Lemma~\ref{lem:dinolo} to $D$. By Observation~\ref{obs:reverse_mono}.2 we get:

\begin{lem}
If $G=(V,E)$ is subgraph of a minimal Cayley graph $\Cay(\Gamma,S)$, then assigning via $f:E\to S$ to each $e\in E$ the corresponding element of $S$ yields a faithful no-lonely-color.    
\end{lem} 

A useful reformulation of Definition~\ref{def:lonely}.2 is the following:

\begin{obs}\label{obs:cut}
Let $f:E\to  S$ be an edge-coloring of $G=(V,E)$. Then $f$ satisfies Definition~\ref{def:lonely}.2 if and only if for any $i\in  S$ and $e=\{u,v\}\in f^{-1}(i)$ the vertices $u,v$ lie in different connected components of $G\setminus f^{-1}(i)$.
\end{obs}

\section{An algorithm}
The implementation of the following algorithms, is available in our GitHub repository~\cite{code_repo}.

First, we devise an algorithm that given a graph $G=(V,E)$  produces all faithful no-lonely-color colorings. We extend a partial coloring $f$. We assume the set $C=\{1, \ldots, k\}$ for some $k\in \mathbb{N}$. Initially however we set, $f(e)=0$ for all $e\in E$, which stands for $e$ being uncolored.


\begin{algorithm}[H]
    \caption{\textsc{extendColoring}: Extends a partial faithful no-lonely-color coloring}
    \label{alg:all_no_lonely_colorings}
    \SetAlgoLined

    \textbf{Input:}\\
    \quad $c$: the color currently being extended.\\
    \quad $f\_valid$: \texttt{True} if $f^{-1}(c)$ satisfies Definition~\ref{def:lonely}, \texttt{False} otherwise.\\

    \BlankLine

    \If{all edges are colored \textbf{and} $f\_valid$}{
        \textbf{Output} current coloring $f$\\
        \Return \texttt{True}
    }

    \BlankLine

    \tcp{Determine candidate edges to color next}
    \uIf{there exists a cycle with a unique uncolored edge $e$}{
        $next\_edges \gets \{e\}$
    }
    \Else{
        \uIf{$f\_valid$}{
            \tcp{Start coloring with a new color $c+1$}
            Choose any uncolored edge $e$\\
            $f(e) \gets c+1$\\
            \If{\textsc{extendColoring}($c+1$, \texttt{False})}{
                \Return \texttt{True}
            }
            \tcp{Try other uncolored edges with current color}
            $next\_edges \gets$ uncolored edges of $G$
        }
        \Else{
            \tcp{Fix a cycle with exactly one edge of color $c$}
            Find a cycle $C$ with a unique edge of color $c$\\
            \uIf{such a cycle $C$ exists}{
                $next\_edges \gets$ uncolored edges of $C$
            }
            \Else{
                $next\_edges \gets$ uncolored edges of $G$
            }
        }
    }

    \BlankLine

    \ForEach{$e \in next\_edges$}{
        $f(e) \gets c$\\

        \If{$\Delta(f^{-1}(c)) \leq 2$ \textbf{and} all cycles in $f^{-1}(c)$ have the same length}{
            \BlankLine
            \tcp{Check Observation~\ref{obs:cut}}
            \If{$u', v'$ are in different components of $G \setminus f^{-1}(c)$ for all $e' \in f^{-1}(c)$}{
                \If{\textsc{extendColoring}($c$, \texttt{True})}{
                    \Return \texttt{True}
                }
            }
            \If{\textsc{extendColoring}($c$, \texttt{False})}{
                \Return \texttt{True}
            }
        }
        \Else{
            $f(e) \gets 0$ \tcp*{Reset edge if invalid}
        }
    }

    \BlankLine

    \Return \texttt{False}
\end{algorithm}

After having generated all faithful no-lonely-color colorings of $G$ we reduce the list by removing colorings that can be obtained from others by applying some automorphism of $G$ or by relabeling colors. We then proceed to find orientations of $G$ that satisfy Lemma~\ref{lem:dinolo} and in particular Observation~\ref{obs:matching}.

For the following assume we are given $G=(V,E)$ together with a no-lonely-color coloring $f:E\to  S$. We devise an algorithm to construct an orientation $D$ of $G$ color by color to obtain a candidate orientation. It is based on the following easy observations. For this, we note that if $H$ is connected and \emph{non-trivial}, i.e., with more than one vertex of maximum degree $\leq 2$. Then $H$ has exactly two orientations satisfying Definition~\ref{def:dilonely}.1, denoted by $H^+$ and $H^-$. By Observation~\ref{obs:reverse_mono} in each color, one of these orientations can be arbitrarily fixed. Further, note that by Observation~\ref{obs:matching}, only if some color $f^{-1}(i)$ is a matching, then its edges can potentially be completely bi-oriented. 
We initialize $D=G$ unoriented and $c=1$.

\begin{algorithm}[H]
\caption{\textsc{extendCandidateOrientation}: Extends candidate orientations of a graph with no-lonely-color $k$-coloring.}
\label{alg:obtainAllPossibleOrientations}
\textbf{Input:} An integer $c$ indicating the next color class to orient.\\

\If{$c = k+1$}{
    \textbf{output} $D$
}
\Else{
    fix an arbitrary non-trivial connected component  $H'$ of $f^{-1}(c)$\\
    \For{$H'\in S \subseteq \{ H \mid H \text{ is a non-trivial connected component of } f^{-1}(c) \}$}{
        Orient each $H \in S$ as $H^+$ in $D$\\
        Orient each $H \notin S$ as $H^-$ in $D$\\
        \textsc{extendCandidateOrientation}$(c+1)$
    }
    \If{$f^{-1}(c)$ is a matching}{
        Bi-orient all edges in $f^{-1}(c)$ in $D$\\
        \textsc{extendCandidateOrientation}$(c+1)$
    }
}
\end{algorithm}

For each of the generated candidate orientations we use the following algorithm that checks the special case of Definition~\ref{def:dilonely}.3, where only cycles are considered.

\begin{algorithm}[H]\label{algo:cycles}
    \caption{\textsc{checkCycles}: Verify if the color sequence of every cycle, when followed from any vertex, yields a cycle} 
    \label{alg:checkCycles}
    \textbf{Input:}\\
    \quad $D=(V,A), f$: A colored directed graph, where $f$ assigns a color to each arc.\\
    \textbf{Output:} \\
    \quad \texttt{True} if, for every cycle $C$ in $D$, the color/direction pattern of $C$ starting at any vertex yields a cycle; \texttt{False} otherwise.
    \vspace{0.3cm}
    
    \For{every cycle $C$ in $D$}{
        Represent $C$ as a sequence of arcs with directions: \\
        \quad $C = (a^{\sigma(1)}_1, \ldots, a^{\sigma(\ell)}_\ell)$ \\
        where each $a_i$ is an arc and $\sigma(i) \in \{+1, -1\}$ indicates direction.
        \BlankLine
        Let $\mathbf{c} = (f(a_1), \ldots, f(a_\ell))$ be the color sequence. \\
        Let $\boldsymbol{\sigma} = (\sigma(1), \ldots, \sigma(\ell))$ be the direction sequence.
        \BlankLine
        \For{every vertex $v$ in $V$}{
            Try to follow a path from $v$ using color sequence $\mathbf{c}$ and direction sequence $\boldsymbol{\sigma}$. \\
            \If{a path is found, but it does not return to $v$ or self-intersects}{
                \Return \texttt{False}
            }
        }
    }
    \Return \texttt{True}
\end{algorithm}
\section{Two cubic graphs of girth 5}
There are two cubic graphs of girth 5. Our implementation~\cite{code_repo} applied to the \emph{triplex graph}, see the left of Figure~\ref{fig:triplex}, shows that it is not the subgraph of a minimal Cayely graph. Indeed, Algorithm \ref{alg:all_no_lonely_colorings} produces all 216 faithful no-lonely color colorings associated to this graph, in 11 seconds. After reducing modulo automorphism, we are only left with 15 colorings. Algorithm~\ref{alg:obtainAllPossibleOrientations} and Algorithm~\ref{algo:cycles} finish without finding any orientation to any of these colorings in about 45 seconds.

\begin{figure}
    \centering
    \includegraphics[width=\linewidth]{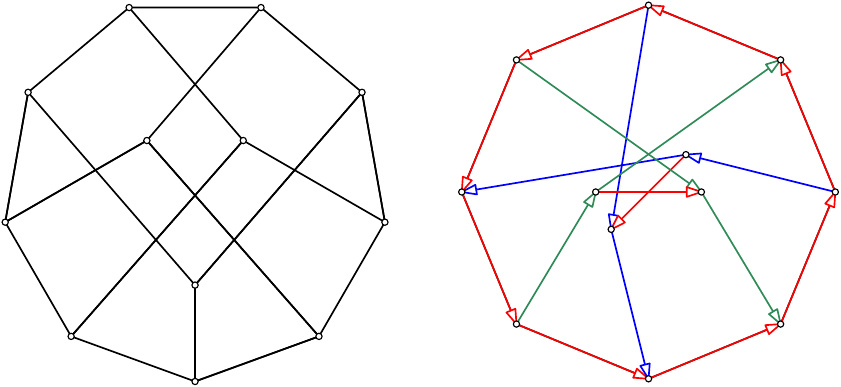}
    \caption{Left: The triplex graph is not subgraph of a minimal Cayley graph. Right: An orientation of the twinplex graph with a coloring from our algorithm.}
    \label{fig:triplex}
\end{figure}

The second cubic graphs of girth 5 is the \emph{twinplex graph} depicted on the right of Figure~\ref{fig:triplex}. In this case, our procedure finds an orientation and a coloring that passes our algorithm, as depicted in the figure. Hence, we cannot conclude whether the graph is a subgraph of a minimal Cayley graph or not.

\section{Generalized Petersen Graphs}

Let $k,n$ be integers such that $0<k<\frac{n}{2}$. The \emph{generalized Petersen graph} $G(n,k)$ is the cubic graph on vertex set $V = V_I \cup V_O$, where $V_I = \{v_0, \ldots, v_{n-1}\}$ is the set of {\it inner vertices} and $V_O = \{u_0,\ldots, u_{n-1}\}$ the set of {\it outer vertices}. The edge set is partitioned into three parts (all subscripts are considered modulo $n$): the edges $E_O(n,k) = \{u_i u_{i+1} \, \vert \, 0 \leq i \leq n-1\}$ inducing a cycle of length $n$; the edges $E_I(n,k) = \{v_i v_{i+k} \, \vert \, 0 \leq i \leq n-1\}$, inducing $\gcd(n,k)$ cycles of length $\frac{n}{\gcd(n,k)}$; and the edges $E_S(n,k) = \{u_i v_i \, \vert \, 0 \leq i \leq n-1\}$ forming a perfect matching of $G(n,k)$. Generalized Petersen graphs were introduced by Coxeter~\cite{Cox-50} and named by Watkins~\cite{Watkins:69}. 
Many algebraic properties of $G(n,k)$ depend on the particular $k,n$, e.g., isomorphisms~\cite{Ste-09}, automorphism groups, edge-and vertex-transitivity~\cite{FGW:71}, being a core~\cite{GMK24}, and being a Cayley graph~\cite{Ned-95,Lov-97}.

\begin{figure}[h]
\centering
\includegraphics[width=\textwidth]{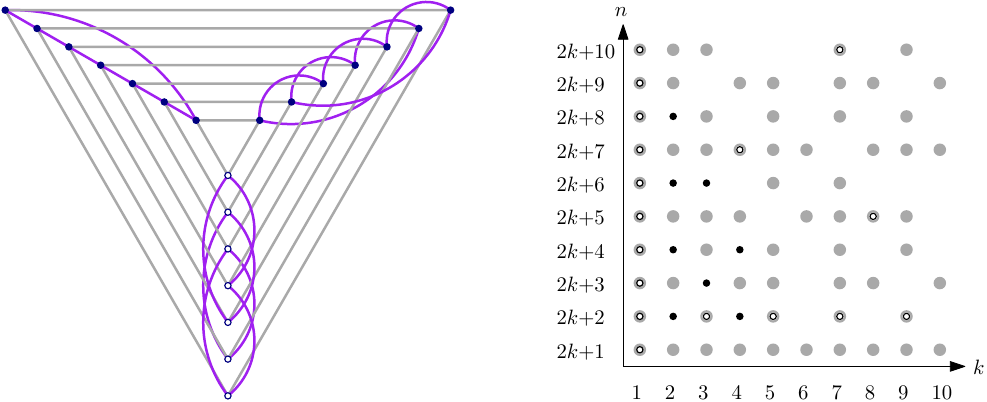}
\caption{Left: the filled vertices  induce $G(7,2)$ in $\Cay(C_7\rtimes C_3,\{a,b\})$. Right: The Petersen plane, where white dots are Cayley graphs, Grey dots are induced subgraph of a minimal Cayley graph of $C_n\rtimes C_r$ by Theorem~\ref{thm:Petersen}, and black dots have been obtained by computer. }\label{fig:G72inC3C7}
\end{figure}

\begin{thm}\label{thm:Petersen}
Let $k,r,n\in\mathbb{N}$ such that $\gcd(k,n)=1$ and  $k^r=1\mod n$. Then $G(n,k)$ is an induced subgraph of a minimal Cayley graph of $C_n\rtimes C_r$.
\end{thm}
\begin{proof}
 It is known that since $\gcd(k,n)=1$ and $k^r=1\mod n$ the semidirect product $C_n\rtimes C_r$ has the (minimal) presentation: $$<a,b\mid a^r=b^n=e, aba^{-1}=b^k>.$$
 We show that $G(n,k)$ is isomorphic to an induced subgraph of $\Cay(C_n\rtimes C_r,\{a,b\})$. Indeed, for $0\leq i\leq n-1$ map every outer vertex $u_i$ to $b^i$ and every inner vertex $v_i$ to $b^ia$. See the left side of  Figure~\ref{fig:G72inC3C7} for an example. Clearly, by minimality of the generating set, the outer vertices induce a cycle of length $n$ in $\Cay(C_n\rtimes C_r,\{a,b\})$. Further, there is an edge between $u_i$ and $v_j$ if and only if $b^{-i}b^ja\in\{a,b\}$ or $a^{-1}b^{-j}b^{i}\in\{a,b\}$. This is the case if and only if $i=j$. Now, consider inner vertices $v^iv^j$ these have an edge if $a^{-1}b^{-i}b^ja\in\{a,b\}$ or $a^{-1}b^{-j}b^{i}a\in\{a,b\}$. Using $aba^{-1}=b^k$ one gets that this is the case if and only if $i-j=k\mod n$, which finishes the proof.
\end{proof}

Note that the characterization of generalized Petersen graphs, that are minimal Cayley graphs~\cite{Ned-95,Lov-97} yields a subset of the graphs of Theorem~\ref{thm:Petersen}.
Apart from the above result, we further checked for some small parameters if a given generalized Petersen graph is induced subgraph of a minimal Cayley graph of a group. For this we used the lists of such graphs computed independently by Kolja Knauer~\cite{CDG23}
and Rhys Evans~\cite{Rhys}. We display the results in the right side of Figure~\ref{fig:G72inC3C7}.


\section{Questions}

Our work leaves two principal questions open:

\begin{conj}
For every integer $g\geq 3$ there is a connected cubic graph of girth $g$ that is not a subgraph of a minimal Cayley graph.
\end{conj}

\begin{conj}
Every generalized Petersen graph is an induced subgraph of a minimal Cayley graph.    
\end{conj}

\paragraph{Acknowledgments:} KK was supported by the MICINN grant PID2022-137283NB-C22 and by the Spanish State Research Agency, through the Severo Ochoa and María de Maeztu Program for Centers and Units of Excellence in R\&D (CEX2020-001084-M).
\small
\bibliography{fraglit}

\begin{thebibliography}{10}

\bibitem{B78}
{\sc L.~{Babai}}, {\em {Chromatic number and subgraphs of {C}ayley graphs}}.
\newblock {Theor. Appl. Graphs, Proc. Kalamazoo 1976, Lect. Notes Math. 642,
  10-22 (1978).}

\bibitem{Bab78}
{\sc L.~Babai}, {\em Embedding graphs in {Cayley} graphs}.
\newblock Probl{\`e}mes combinatoires et th{\'e}orie des graphes, {Orsay} 1976,
  {Colloq}. int. {CNRS} {No}. 260, 13-15 (1978).

\bibitem{B95}
{\sc L.~{Babai}}, {\em {Automorphism groups, isomorphism, reconstruction}}, in
  {Handbook of combinatorics. Vol. 1-2}, Amsterdam: Elsevier (North-Holland);
  Cambridge, MA: MIT Press, 1995, pp.~1447--1540.

\bibitem{BS85}
{\sc L.~Babai and V.~T. S{\'o}s}, {\em Sidon sets in groups and induced
  subgraphs of {Cayley} graphs}, Eur. J. Comb., 6 (1985), pp.~101--114.

\bibitem{CDG23}
{\sc K.~Coolsaet, S.~D'hondt, and J.~Goedgebeur}, {\em House of {Graphs} 2.0: a
  database of interesting graphs and more}, Discrete Appl. Math., 325 (2023),
  pp.~97--107.

\bibitem{Cox-50}
{\sc H.~S.~M. {Coxeter}}, {\em {Self-dual configurations and regular graphs.}},
  {Bull. Am. Math. Soc.}, 56 (1950), pp.~413--455.

\bibitem{DHY24}
{\sc J.~Davies, M.~Hatzel, and L.~Yepremyan}, {\em Counterexample to {Babai}'s
  lonely colour conjecture}.
\newblock Preprint, {arXiv}:2410.05199 [math.{CO}] (2024), 2024.

\bibitem{FM16}
{\sc M.~Farrokhi and A.~Mohammadian}, {\em Groups whose all (minimal) {C}ayley
  graphs have a given forbidden structure (research on finite groups and their
  representations, vertex operator algebras, and algebraic combinatorics)},
  Kyoto University Research Information Repository, 2003 (2016).

\bibitem{FGW:71}
{\sc R.~Frucht, J.~E. Graver, and M.~E. Watkins}, {\em The groups of the
  generalized {P}etersen graphs}, Proc. Cambridge Philos. Soc., 70 (1971),
  pp.~211--218.

\bibitem{GMK22}
{\sc I.~Garc{\'{\i}}a-Marco and K.~Knauer}, {\em On sensitivity in bipartite
  {Cayley} graphs}, J. Comb. Theory, Ser. B, 154 (2022), pp.~211--238.

\bibitem{GMK24}
{\sc I.~Garc{\'{\i}}a-Marco and K.~Knauer}, {\em Beyond symmetry in generalized
  {Petersen} graphs}, J. Algebr. Comb., 59 (2024), pp.~331--357.

\bibitem{GMK25}
{\sc I.~Garc{\'{\i}}a-Marco and K.~Knauer}, {\em Coloring minimal {Cayley}
  graphs}, Eur. J. Comb., 125 (2025), p.~9.
\newblock Id/No 104108.

\bibitem{God81}
{\sc C.~D. Godsil}, {\em Connectivity of minimal {Cayley} graphs}, Arch. Math.,
  37 (1981), pp.~473--476.

\bibitem{GI87}
{\sc C.~D. Godsil and W.~Imrich}, {\em Embedding graphs in {Cayley} graphs},
  Graphs Comb., 3 (1987), pp.~39--43.

\bibitem{HI96}
{\sc A.~Hassani and M.~A. Iranmanesh}, {\em On {Cayley} graphs related to
  finite groups with minimal generator set}, in Proceedings of the 27th annual
  Iranian mathematics conference, Shiraz, Iran, March 27--30, 1996, Shiraz:
  Shiraz University, Dept. of Mathematics, 1996, pp.~107--111.

\bibitem{KK19}
{\sc U.~Knauer and K.~Knauer}, {\em Algebraic graph theory. {Morphisms},
  monoids and matrices}, vol.~41 of De Gruyter Stud. Math., Berlin: De Gruyter,
  2nd revised and extended edition~ed., 2019.

\bibitem{LZ01}
{\sc C.~H. Li and S.~Zhou}, {\em On isomorphisms of minimal {Cayley} graphs and
  digraphs}, Graphs Comb., 17 (2001), pp.~307--314.

\bibitem{Lov-97}
{\sc M.~{Lovre\v{c}i\v{c} Sara\v{z}in}}, {\em {A note on the generalized
  Petersen graphs that are also Cayley graphs}}, {J. Comb. Theory, Ser. B}, 69
  (1997), pp.~226--229.

\bibitem{MS20}
{\sc {\v{S}}.~Miklavi{\v{c}} and P.~{\v{S}}parl}, {\em On minimal
  distance-regular {Cayley} graphs of generalized dihedral groups}, Electron.
  J. Comb., 27 (2020), pp.~research paper p4.33, 16.

\bibitem{Ned-95}
{\sc R.~{Nedela} and M.~{\v{S}koviera}}, {\em {Which generalized Petersen
  graphs are Cayley graphs?}}, {J. Graph Theory}, 19 (1995), pp.~1--11.

\bibitem{PR09}
{\sc I.~Pak and R.~Radoi{\v{c}}i{\'c}}, {\em Hamiltonian paths in {Cayley}
  graphs}, Discrete Math., 309 (2009), pp.~5501--5508.

\bibitem{Rhys}
{\sc P.~Potočnik, G.~Potočnik, and R.~Evans}, {\em Datasets of highly
  symmetric objects}.
\newblock \url{https://graphsym.net/}, 2024.
\newblock Accessed: April 17, 2025.

\bibitem{SS09}
{\sc A.~J. Slupik and V.~I. Sushchansky}, {\em Minimal generating sets and
  {Cayley} graphs of {Sylow} {{\(p\)}}-subgroups of finite symmetric groups.},
  Algebra Discrete Math., 2009 (2009), pp.~167--184.

\bibitem{code_repo}
{\sc {\'A}.~Soto~Gomez}, {\em Code for "{What is and is not inside a Cayley
  graph?}"}.
\newblock \url{https://github.com/alvarosg285/on_subgraphs_of_cayley_graphs},
  2025.

\bibitem{Spe83}
{\sc J.~Spencer}, {\em What's not inside a {Cayley} graph}, Combinatorica, 3
  (1983), pp.~239--241.

\bibitem{Ste-09}
{\sc A.~{Steimle} and W.~{Staton}}, {\em {The isomorphism classes of the
  generalized Petersen graphs}}, {Discrete Math.}, 309 (2009), pp.~231--237.

\bibitem{Watkins:69}
{\sc M.~E. Watkins}, {\em A theorem on {T}ait colorings with an application to
  the generalized {P}etersen graphs}, J. Combinatorial Theory, 6 (1969),
  pp.~152--164.

\bibitem{Whi01}
{\sc A.~T. White}, {\em Graphs of groups on surfaces. {Interactions} and
  models}, vol.~188 of North-Holland Math. Stud., Amsterdam: Elsevier, 2001.

\end{thebibliography}
\bibliographystyle{my-siam}

\end{document}